\documentclass[11pt]{article}
\RequirePackage{amsmath,amsthm,amsfonts,dsfont}
\RequirePackage{amssymb,mathrsfs}

\RequirePackage{fullpage}
\RequirePackage[defaultlines=3,all]{nowidow} 
\RequirePackage{lmodern}
\usepackage[utf8]{inputenc}
\usepackage[T1]{fontenc}
\RequirePackage{microtype}
\RequirePackage{enumitem}
\RequirePackage{mathtools}
\mathtoolsset{centercolon}  
\RequirePackage{bm} 
\RequirePackage[font=small,labelfont=bf]{caption}
\RequirePackage{booktabs}
\RequirePackage{graphicx}
\RequirePackage{mleftright}
\mleftright

\RequirePackage{xcolor}
\RequirePackage{hyperref}
\hypersetup{
        breaklinks=true,
        colorlinks,
	linkcolor={blue!65!black},
	citecolor={blue!65!black},
	urlcolor={red!65!black},
        filecolor={red!65!black},
}
\newcommand{\email}[1]{\texttt{\href{mailto:#1}{#1}}}
\RequirePackage[capitalize,nameinlink,noabbrev,nosort]{cleveref} 
\usepackage[hyperpageref]{backref} \newcommand{\arXiv}[1]{\href{https://arxiv.org/abs/#1}{arXiv:#1}}
\newcommand{\arxiv}[1]{\arXiv} 

\newtheorem{theorem}{Theorem}[section]
\newtheorem{remark}[theorem]{Remark}

\newtheorem{corollary}[theorem]{Corollary}
\newtheorem{definition}[theorem]{Definition}

\newcommand{\C}{\ensuremath{\mathbb{C}}}

\newcommand{\F}{\ensuremath{\mathbb{F}}}
\newcommand{\R}{\ensuremath{\mathbb{R}}}
\newcommand{\Z}{\ensuremath{\mathbb{Z}}}

\newcommand{\lat}{\mathcal{L}}

\newcommand{\eps}{\varepsilon}

\DeclareMathOperator{\spn}{span} 
 
\newcommand{\lspan}{\spn}

\renewcommand{\epsilon}{\eps} 
\renewcommand{\vec}[1]{\bm{#1}}

\renewcommand{\bar}{\overline}

\DeclarePairedDelimiter\set{\{}{\}}
\DeclarePairedDelimiter\norm{\lVert}{\rVert}
\newcommand{\card}[1]{\left| {#1} \right|}
\newcommand{\bit}{\set{0, 1}}

\RequirePackage{xspace}

\usepackage{fullpage}

\renewcommand{\C}{\mathcal{C}}

\newcommand{\Vladut}{Vl\u{a}du\c{t}\xspace}

\newcommand{\supp}{\mathrm{supp}}

\renewcommand{\bar}{\overline}

\title{Difficulties Constructing Lattices with Exponential Kissing Number from Codes}
\author{Huck Bennett\thanks{University of Colorado Boulder, \email{huckbennett@gmail.com}. This work was supported by NSF Award No. CCF-2432132.} \and Alexander Golovnev\thanks{Georgetown University, \email{alexgolovnev@gmail.com}. This work was supported by the NSF CAREER award (grant CCF-2338730).} \and Noah Stephens-Davidowitz\thanks{Cornell University, \email{noahsd@gmail.com}. This work was supported by NSF Grants Nos.~CCF-2122230 and CCF-2312296, a Packard Foundation Fellowship, and a generous gift from Google.}}
\date{\today}

\begin{document}

\maketitle

\begin{abstract}
In this note, we present examples showing that several natural ways of constructing lattices from error-correcting codes do not in general yield a correspondence between minimum-weight non-zero codewords and shortest non-zero lattice vectors. From these examples, we conclude that the main results in two works of Vl\u{a}du\c{t} (Moscow J. Comb. Number Th., 2019 and Discrete Comput. Geom., 2021) on constructing lattices with exponential kissing number from error-correcting codes are invalid. A more recent preprint (arXiv, 2024) that Vl\u{a}du\c{t} posted after an initial version of this work was made public is also invalid. 

Exhibiting a family of lattices with exponential kissing number therefore remains an open problem (as of July 2025).
\end{abstract}

\section{Introduction}
A (linear, binary) \emph{code} $\C = \C(G) \subseteq \F_2^n$ of dimension $k$ is defined as $\C(G) := \set{\sum_{i=1}^k a_i \vec{g}_i : a_1, \ldots, a_k \in \F_2^n}$, where the generator matrix $G = (\vec{g}_1, \ldots, \vec{g}_k) \in \F_2^{n \times k}$ has linearly independent columns.\footnote{In this note, we use column bases both for codes and lattices.} A~\emph{lattice} $\lat = \lat(B) \subset \R^n$ of rank $k$ is defined as $\lat(B) := \set{\sum_{i=1}^k a_i \vec{b}_i : a_1, \ldots, a_k \in \Z}$, where the basis $B = (\vec{b}_1, \ldots, \vec{b}_k) \in \R^{n \times k}$ has linearly independent columns. We extend this notation and define $\C(G_0,\ldots, G_a)$ to be the $\F_2$-linear span of the columns of the matrices $G_0,\ldots, G_a$ and $\lat(B_0, \ldots, B_a)$ to be the integer linear span of the columns of the (integer) matrices $B_0, \ldots, B_a$, and we do not insist on linear independence.

The \emph{minimum distance} of a code $\C \subseteq \F_2^n$ is 
$d(\C) := \min_{\vec{c} \in \C \setminus \set{\vec{0}}} \norm{\vec{c}}_0$,
where $\norm{\vec{x}}_0$ denotes the Hamming weight of $\vec{x} \in \F_2^n$, and similarly the (Euclidean) minimum distance of a lattice $\lat \subset \R^n$ is 
$\lambda_1(\lat) := \min_{\vec{y} \in \lat \setminus \set{\vec{0}}} \norm{\vec{y}}_2$,
where $\norm{\vec{x}}_2$ denotes the Euclidean norm of $\vec{x} \in \R^n$.
The \emph{kissing number} $\kappa_0(\C)$ of a code $\C$ is the number of minimum-weight non-zero codewords, i.e., 
\[
\kappa_0(\C) := \card{\set{\vec{c} \in \C : \norm{\vec{c}}_0 = d(\C)}}
\; .
\]
Similarly, the (Euclidean) \emph{kissing number} $\kappa_2(\lat)$ of a lattice $\lat$ is the number of shortest non-zero vectors, i.e., 
\[
    \kappa_2(\lat) := |\set{\vec{y} \in \lat : \norm{\vec{y}}_2 = \lambda_1(\lat)}|
    \; .
\]
(One can also define the minimum distance and the kissing number for arbitrary norms.)

For a vector $\vec{x} \in \F_2^n$, we write $\bar{\vec{x}} \in \{0,1\}^n \subset \Z^n$ to denote the natural entry-wise \emph{embedding} of $\vec{x}$ into $\Z^n$ (i.e., $0 \in \F_2$ is mapped to $0 \in \Z$ and $1 \in \F_2$ is mapped to $1 \in \Z$). We extend this notation so that it applies to matrices $A \in \F_2^{n \times m}$ and subsets $S \subseteq \F_2^n$ in the natural way. 

\paragraph{Codes and lattices with exponential kissing number.} Kalai and Linial~\cite{kalaiDistanceDistributionCodes1995} asked in 1995 whether there is an infinite family of linear codes $\set{\C_n}$ with $\C_n \subseteq \F_2^n$ and $\kappa_0(\C_n) \geq 2^{\Omega(n)}$, and they conjectured that no such family of codes exists (i.e., that the kissing number must be subexponential in~$n$). This conjecture was disproven in 2001, when Ashikhmin, Barg, and \Vladut succeeded in constructing such a family based on algebraic-geometry codes~\cite{Ashikhmin-Barg-Vladut-Exp-Kiss-Codes-01}. 

Similarly, it had been a noted open problem to construct an infinite family $\set{\lat_n}$ of lattices with $\lat_n \subset \R^n$ and $\kappa_2(\lat_n) \geq 2^{\Omega(n)}$ (see, e.g.,~\cite{alonPackingsLargeMinimum1997}).  Recently, \cite{Vladut-l2kissnum-2019} gave a construction of a family of lattices from certain codes, and claimed that the resulting family of lattices has exponential kissing number if the underlying codes do. \cite{Vladut-lpkissnum-2021} then showed a different construction and claimed that it could be used to generate lattices that have exponential kissing number in $\ell_p$ norms (and certain other norms).

Unfortunately, the proofs in~\cite{Vladut-l2kissnum-2019,Vladut-lpkissnum-2021} that the resulting lattices have exponential kissing number do not work. In both cases, the proofs fail due to the fact that certain natural constructions of lattices from codes do not yield a correspondence between minimum-weight non-zero codewords and shortest non-zero lattice vectors. This arises from the fundamental but easy-to-forget fact that for $\vec{x}_1, \vec{x}_2 \in \F_2^n$, \emph{the sum $\bar{\vec{x}}_1 + \bar{\vec{x}}_2$ over $\Z^n$ of the embeddings is not necessarily the same as the embedding $\bar{\vec{x}_1 + \vec{x}_2}$ of the sum $\vec{x}_1 + \vec{x}_2$ over $\F_2$}. 

After a preliminary version of this work was published, \Vladut posted a new paper to arXiv \cite{vladut24} attempting to correct the errors in~\cite{Vladut-l2kissnum-2019,Vladut-lpkissnum-2021}.
However, similar issues remain and the new proof is also invalid. 

We detail these issues in the sequel. We also note that both~\cite{Vladut-l2kissnum-2019,Vladut-lpkissnum-2021} were retracted on arXiv after we communicated these issues to \Vladut~\cite{vladut-com-24}. As of July 2025,~\cite{vladut24} has not yet been retracted.
Consequently, the problem of constructing a family of lattices with exponential kissing number remains a tantalizing open problem. %

\paragraph{Applications of lattices with large kissing number.} The existence of an infinite family of lattices~$\{\lat_n\}$ with exponential kissing number  would imply certain complexity-theoretic hardness results for computational lattice problems~\cite{conf/stoc/AggarwalS18,conf/innovations/BennettPT22}. (In fact, even families $\set{\lat_n}$ satisfying $\kappa_2(\lat_n) \geq 2^{n^{\eps}}$ for any constant $\eps > 0$ would be useful for such purposes. As far as we know, the best known construction yields $\kappa_2(\lat_n) = n^{\Theta(\log n)}$~\cite{barnesExtremeFormsDefined1959}.) Indeed, some of the hardness results for computational lattice problems in \cite{conf/stoc/AggarwalS18,conf/innovations/BennettPT22} are only known under the still unproven assumption that such a family of lattices exists (though they were claimed unconditionally in \cite{conf/innovations/BennettPT22}, citing \cite{Vladut-l2kissnum-2019}).

Lattices with large kissing number might also yield good sphere packings, which have many applications throughout mathematics and computer science. (See, e.g., \cite{conway1999sphere}.)

\paragraph{Acknowledgments.}
The authors are grateful to Serge \Vladut~\cite{vladut-com-24} for very helpful discussions related to~\cite{Vladut-l2kissnum-2019,Vladut-lpkissnum-2021,vladut24} and the issues we bring up in this note. Some of this work was completed while the authors were visiting Divesh Aggarwal, the Centre for Quantum Technologies, and the National University of Singapore. We would like to thank them for hosting us. Finally, we thank the anonymous reviewers and Fred the bird.

\section{Lattices Constructed from Codes}
The strategy used in \cite{Vladut-l2kissnum-2019,Vladut-lpkissnum-2021,vladut24} is to construct lattices from (a tower of) codes with exponential kissing number, and the issues with the claims in \cite{Vladut-l2kissnum-2019,Vladut-lpkissnum-2021,vladut24} arise from certain subtleties in these constructions that we discuss below. 
The rough idea behind all of these constructions is to produce a lattice $\lat \subseteq \Z^n$ from a code $\C \subseteq \F_2^n$ such that $\lambda_1(\lat)^2 \geq d(\C)$ (and such that the determinant of the lattice is small, provided that the code has large dimension).
We refer the reader to~\cite{conway1999sphere} for much more discussion of these and other constructions of lattices from codes.  The construction we discuss in \cref{sec:all-min-codewords} does not appear as a named construction in~\cite{conway1999sphere} and does not seem useful for sphere packing.

For our purposes, it suffices to note that such constructions can be made to work for constructing good sphere packings. However, we present examples showing that these constructions do not in general yield a direct correspondence between minimum-weight non-zero codewords in the (tower of) codes and shortest non-zero vectors in the resulting lattice. In particular, these constructions do not obviously yield lattices with large kissing number, even if we start with codes with large kissing number.

\subsection{Construction A} 
\label{sec:cons-A}

Before we present the more sophisticated constructions that are the focus of this note, we first consider what is perhaps the simplest way to construct a lattice $\lat$ from a (single) code $\C \subseteq \F_2^n$. This is Construction A, which defines a lattice $\lat_A$ from the code $\C$ as
\[
\lat_A = \lat_A(\C) := \bar{\C} + 2\Z^n = \set{\vec{z} \in \Z^n : (\vec{z} \bmod 2) \in \C} \ \text{.}
\]
Here $\bar{\C} := \{ \bar{\vec{c}} \ : \ \vec{c} \in \C\}$, where $\bar{\vec{c}} \in \{0,1\}^n$ is the natural coordinate-wise embedding of an element $\vec{c} \in \F_2^n$ into $\{0,1\}^n$.
Notice that for every codeword $\vec{c} \in \C$, the embedded vector $\bar{\vec{c}} \in \{0,1\}^n \subset \Z^n$ is in $\lat_A$, making this construction quite natural.
However, because $\lat_A$ contains the $2n$ vectors $\pm 2 \vec{e}_i$ (where $\vec{e}_i$ is the $i$th standard normal basis vector), $\lambda_1(\lat_A)^2 = \min(d(\C),\, 4) < d(\C)$ whenever $d(\C) \geq 5$. 
In particular, whenever $d(\C) \geq 5$, the shortest non-zero vectors in $\lat_A$ are simply the shortest non-zero vectors in $2\Z^n$. Construction A is therefore not particularly useful for constructing lattices with large minimum distance or large kissing number.

\subsection{Construction D}

We next describe Construction D, which generalizes Construction A and is motivated in part by the failure of Construction A to produce lattices that maintain the minimum distance of the underlying code 
(see \cite[Ch. 8, Sec. 8]{conway1999sphere} and \cite[Sec. 4]{journals/toc/Micciancio12}, which we mostly follow here). Construction D is defined as follows.\footnote{The definition of Construction D that we use here is from~\cite{journals/toc/Micciancio12}. It produces scalings of the Construction-D lattices defined in~\cite{conway1999sphere}.}

\begin{definition}[Construction D] \label{def:consD}
Let $\F_2^n = \C_0 \supseteq \C_1 \supseteq \cdots \supseteq \C_a$ be a tower of codes where $\C_i$ has dimension $k_i$ and minimum distance $d_i \geq 4^i$ for $i = 0, \ldots, a$. 
Let $K_0, \ldots, K_a$ be matrices such that $K_i \in \F_2^{n \times (k_i - k_{i+1})}$ (where we define $k_{a + 1} := 0$) and such that $(K_i, \ldots, K_a)$ is a generator matrix for $\C_i$ for $i = 0, \ldots, a$.
The Construction-$D$ lattice obtained from $K_0, \ldots, K_a$ is then defined as
\begin{equation*} %
\lat_D = \lat_D(K_0, \ldots, K_a)
:= \lat(2^a \cdot \bar{K}_0, 2^{a-1} \cdot \bar{K}_1, \ldots, 2 \cdot \bar{K}_{a-1}, \bar{K}_a) \ \text{.}
\end{equation*}
\end{definition}

We make several remarks about this construction.
\begin{enumerate}
\item 
Most sources (e.g.,~\cite{conway1999sphere}) require that $G_0 := (K_0, \ldots, K_{a}) \in \F_2^{n \times n}$ be (lower or upper) triangular, in which case $\bar{G}_0$ is a basis of $\Z^n$ and 
\[
\lat_D = 2^a \Z^n + \lat(2^{a-1} \cdot \bar{K}_1, \ldots, 2 \cdot \bar{K}_{a-1}, \bar{K}_a) \ \text{.}
\]
Furthermore, when this holds, the $a = 1$ special case of Construction D is essentially Construction A (up to the lower bound requirement on $d_1$ in Construction D).
However,~\cite{Vladut-l2kissnum-2019} does not assume that $G_0$ is triangular, and we will not need this assumption for our counterexample.
\item \label{item:consD-not-unique} The lattice $\lat_D$ depends on the matrices $K_i$ and not only the codes $\C_i$ generated by $(K_i, \ldots, K_a)$. In particular, the embedding $\bar{\vec{c}} \in \bit^n$ of a codeword $\vec{c} \in \C_a$ is not obviously contained in $\lat$ unless $\vec{c}$ is a column of $K_a$. This is the key issue with the use of Construction~D in~\cite{Vladut-l2kissnum-2019}. See~\cite[Remark 4.3]{Mook-Peikert-Decoding-22} for an example of this (though over $\F_3$ rather than $\F_2$). We also give a counterexample to the use of Construction D in \cite{Vladut-l2kissnum-2019,Vladut-lpkissnum-2021} related to this issue below.
\item It holds that $\lambda_1(\lat_D) = 2^a$ (see, e.g.,~\cite[Theorem 5.1 and Footnote 6]{Mook-Peikert-Decoding-22}). So, \emph{if} $\vec{c} \in \C_a$ is a codeword in $\C_a$ of Hamming weight $4^a$ and its embedding $\bar{\vec{c}}$ is in $\lat_D$, then $\norm{\bar{\vec{c}}}_2 = \lambda_1(\lat_D) = 2^a$, i.e., $\bar{\vec{c}}$ is a shortest non-zero vector in $\lat_D$. However, note that by \cref{item:consD-not-unique}, even if $\vec{c} \in \C_a$ it is not necessarily the case that $\bar{\vec{c}} \in \lat_D$. 
\end{enumerate}

\paragraph{The special case of Construction D from~\cite{Vladut-l2kissnum-2019}.}
\cite[Section 6]{Vladut-l2kissnum-2019} proposes applying Construction D to a tower of codes $\F_2^n = \C_0 \supset \C_1 \supset \cdots \supset \C_a$ where $d(\C_a) = 4^a$ and the dimension~$k_i$ of the code $\C_i$ satisfies $k_i = k_{i+1} + 1$ for $i = 1, \ldots, a - 1$. Specifically, each code $\C_i$ for such $i$ is generated by a generator matrix for $\C_{i+1}$ together with an (arbitrary) additional vector $\vec{c}_i \in \F_2^n$ of Hamming weight $4^i$. So,
\begin{equation} \label{eq:vladut-consD}
\lat_D = \lat_D(K_0, \vec{c}_1,\ldots, \vec{c}_{a-1}, K_a) = \lat(2^a \bar{K}_0, 2^{a-1} \bar{\vec{c}}_{1}, \ldots, 2 \bar{\vec{c}}_{a-1}, \bar{K}_a) \ \text{,}
\end{equation}
where $K_a$ is a generator matrix for $\C_a$ and $K_0$ is such that $G_0 := (K_0, \vec{c}_1, \ldots, \vec{c}_{a-1}, K_a) \in \F_2^{n \times n}$ is full-rank.%
\footnote{More specifically,~\cite{Vladut-l2kissnum-2019} takes $\C_a$ to be a code with exponential kissing number. However, \cite{Vladut-l2kissnum-2019} does not invoke this additional property of $\C_a$, or any other properties of $K_0,\vec{c}_1,\ldots, \vec{c}_{a-1}, K_a$.}

\cite{Vladut-l2kissnum-2019} asserts that when applying \cref{eq:vladut-consD} to a generator matrix $K_a$ of $\C_a$ and vectors $\vec{c}_i$ of this form,
\begin{quote}
``\ldots each minimum weight vector of [$\C_a$] produces a minimum norm vector in [the resulting Construction D lattice $\lat_D$].''
\end{quote}
However, this claim is made without proof and is false. We show how to construct counterexamples in what follows.\footnote{To be falsifiable, the word ``produces'' needs to be defined. In fact, our counterexample rules out even a rather broad interpretation of this, showing that there is no $\vec{y} \in \lat_D \cap \{-1,0,1\}^n$ with $\|\vec{y}\|_2^2 = d(\C_a)$.} The rough idea is to use a gadget matrix $B$ that has a non-trivial kernel over $\F_2$ but is such that $\bar{B}$ has a trivial kernel over the reals.
We use $\vec{0}_k$ and $\vec{1}_k$ to denote the all-zeroes and all-ones vectors of length $k$, respectively, use $I_k$ to denote the $k \times k$ identity matrix, and use $\vec{0}_{k \times \ell}$ to denote the $k \times \ell$ all-zeroes matrix.

\begin{theorem} \label{thm:consD-counterexample}
Let $n_1, n_2 \geq 1$, $a \geq 2$, $m > 4^a$, and $k \geq 1$ be integers. 
Let $A \in \F_2^{n_1 \times k}$, $B \in \F_2^{n_2 \times k}$, and $\vec{w} \in \F_2^k \setminus \set{\vec{0}_k}$ be such that the following all hold.
\begin{enumerate}
\item \label{item:Avec1} $\norm{A \cdot \vec{w}}_0 = 4^a$.
\item \label{item:kerB} $\ker(B) = \set{\vec{0}_{k}, \vec{w}}$.\footnote{We write $\ker(M) := \{\vec{x} \in \F_2^\ell \ : \ M\vec{x} = \vec0\}$ for the kernel of a matrix $M \in \F_2^{n \times \ell}$  \emph{over $\F_2$}.}
\item \label{item:ternaryy} 
For all $\vec{y} \in 2 \Z^k + \bar{\vec{w}}$, we have $(\bar{B} \cdot \vec{y}) \bmod 4 \neq \vec{0}$.
\end{enumerate}

Let
\[
K_a :=
\begin{pmatrix}
A \\
B \otimes \vec{1}_m
\end{pmatrix} \in \F_2^{n \times k} \ \text{,}
\]
where $n := n_1 + m n_2$, and let $\C_a = \C(K_a)$.\footnote{Notice that $K_a$ has full column rank by \cref{item:Avec1,item:kerB}.}
Furthermore, let $\vec{c}_1, \ldots, \vec{c}_{a-1} \in \F_2^{n}$ be vectors such that $\norm{\vec{c}_i}_0 = 4^i$, let $K_0$ be such that $(K_0, \vec{c}_1, \ldots, \vec{c}_{a-1}, K_a) \in \F_2^{n \times n}$ is full-rank, and let $\lat_D = \lat_D(K_0, \vec{c}_1, \ldots, \vec{c}_{a-1}, K_a)$ be as in \cref{eq:vladut-consD}.

Then $d(\C_a) = 4^a$ and $\lat_D \cap \set{-1, 0, 1}^n$ contains no non-zero vectors of norm at most $\sqrt{d(\C_a)} = 2^a$. In particular, for all minimum-weight non-zero codewords $\vec{c} \in \C_a$, we have that $\bar{\vec{c}} \notin \lat_D$.
\end{theorem}

\begin{proof}
Note that for $\vec{x} \in \F_2^{k}$, $\norm{K_a \vec{x}}_0 = \norm{A \vec{x}}_0 + m \cdot \norm{B \vec{x}}_0$. The claim that $d(\C_a) = 4^a$ then follows by noting that $\norm{K_a  \vec{w}}_0 = 4^a$ by \cref{item:Avec1,item:kerB}, and that for $\vec{x} \in \F_2^{k} \setminus \set{\vec{0}_{k}, \vec{w}}$, $ m \cdot \norm{B \vec{x}}_0 \geq m > 4^a$ by \cref{item:kerB}.

We next show that $\lat_D \cap \set{-1, 0, 1}^n$ contains no non-zero vectors of norm at most $2^a$. For $\vec{x} \in \Z^n$, define $\vec{x} \mapsto \vec{x} \bmod 2$ and $\vec{x} \mapsto \vec{x} \bmod 4$ to be the entry-wise maps to $\bit^n$ and $\set{-1, 0, 1, 2}^n$, respectively, and note that for any $\vec{v} \in \Z^n$, $\norm{\vec{v} \bmod 2}_2 \leq \norm{\vec{v} \bmod 4}_2 \leq \norm{\vec{v}}_2$.
Let $\pi$ denote projection onto $\lspan(\set{\vec{e}_i : i \notin \supp(\vec{c}_{a-1})})$, i.e., onto the complement of the support of $\vec{c}_{a-1}$.

Let $\vec{y} \in \lat_D \cap \set{-1, 0, 1}^n$ be a non-zero vector, and note that 
$\vec{y} \in 4 \Z^n + 2 b \bar{\vec{c}}_{a-1} + \bar{K}_a \vec{z}$ for some $b \in \bit$ and $\vec{z} \in \Z^k \setminus (2 \Z^k)$.
We consider two cases. First, consider the case where $\vec{z} \notin 2\Z^k + \bar{\vec{w}}$. Then $\norm{\vec{y}}_2 \geq \norm{\vec{y} \bmod 2}_2 \geq m > 4^a$ by \cref{item:kerB}.
For the second case, we have that $\vec{z} \in 2\Z^k + \bar{\vec{w}}$.
Then,
\begin{align*}
\norm{\vec{y}}_2 &\geq \norm{\vec{y} \bmod 4}_2 \\
                 &= \norm{(2 b \bar{\vec{c}}_{a-1} + \bar{K}_a \vec{z}) \bmod 4}_2 \\
                 &\geq \norm{\pi((2 b \bar{\vec{c}}_{a-1} + \bar{K}_a \vec{z}) \bmod 4)}_2 \\
                 &= \norm{\pi(\bar{K}_a \vec{z} \bmod 4)}_2 \\
                 &\geq \norm{\pi((B \otimes \vec{1}_m) \cdot \vec{z} \bmod 4)}_2 \\
                 &\geq 2 \sqrt{m - \card{\supp(\vec{c}_{a-1})}} \\
                 &> 2 \sqrt{4^a - 4^{a-1}} \\
                 &> 2^a \ \text{,}
\end{align*}
where the fourth inequality holds because \cref{item:ternaryy} implies that $B \vec{z} \bmod 4$ has a coordinate equal to~$2$.
The result follows.
\end{proof}

We now get the following corollary.

\begin{corollary}
There exist a positive integer $n$ and a generator matrix $K_2 \in \F_2^{n \times 3}$ with $d(\C(K)) = 16$ such that for any vector $\vec{c}_1 \in \F_2^n$ of Hamming weight $4$ and any matrix $K_0$ such that $(K_0, \vec{c}_1, K_2) \in \F_2^{n \times n}$ is full-rank, the Construction-D lattice $\lat_D = \lat_D(K_0, \vec{c}_1, K_2)$ (obtained by applying \cref{eq:vladut-consD} with $a = 2$) is such that $\lat_D \cap \set{-1, 0, 1}^n$ contains no vectors of norm $\sqrt{d(\C(K_2))} = 4$. In particular, for all minimum-weight non-zero codewords $\vec{c} \in \C(K_2)$, we have that $\bar{\vec{c}} \notin \lat_D$.
\end{corollary}

\begin{proof}
One can check that the following $A$, $B$, and $\vec{w}$ satisfy the conditions for \cref{thm:consD-counterexample}:
\[
A := \begin{pmatrix}
\vec{1}_{14} & \vec{0}_{14 \times 2} \\
\vec{0}_2 & I_2 \\
\end{pmatrix} 
\ \text{,} \qquad
B := \begin{pmatrix}
1 & 1 & 0 \\
1 & 0 & 1 \\
0 & 1 & 1
\end{pmatrix}
\ \text{,} \qquad
\vec{w} := \vec{1}_3 
\ \text{.}
\]
In particular, one can check that \cref{item:ternaryy} holds as follows. Let $\vec{y} = 2 \vec{z} + \bar{\vec{w}}$ for some $\vec{z} \in \Z^3$, and note that $\vec{1}_3 \notin \C(B)$. Therefore, $\bar{B} \vec{z} \bmod 2$ must have a coordinate equal to $0$, and so $\bar{B} \vec{y} \bmod 4 = (2 \bar{B} \vec{z} + 2 \cdot \vec{1}_3) \bmod 4$ must have a coordinate equal to $2$.
The result follows.
\end{proof}

\begin{remark}
We note that~\cite{Vladut-l2kissnum-2019} also suggests using Construction E, which is a generalization of Construction D. However, the same issue arises there.
\end{remark}

\subsection{Taking All Embeddings of Minimum-Weight Codewords as a Generating Set}
\label{sec:all-min-codewords}

We now consider a different construction, which~\cite{Vladut-lpkissnum-2021} calls ``simplified Construction D.'' However, this construction does not produce a lattice with the same properties as Construction D.

The idea behind this construction is quite simple. Given a code $\C$, we define 
\[
S_{\C} := \{\vec{c} \in \C \ : \ \|\vec{c}\|_0 = d(\C)\}
\]
to be the set of codewords in $\C$ with Hamming weight $d(\C)$ (i.e., the set of shortest non-zero codewords). The lattice constructed in \cite{Vladut-lpkissnum-2021} is then simply 
\[
\lat = \lat(\bar{S}_{\C}) := \{z_1 \bar{\vec{c}}_1 + \cdots + z_{|S_\C|} \bar{\vec{c}}_{|S_\C|} \ : \ z_i \in \Z, \vec{c}_i \in S_{\C}\}
\; ,
\]
i.e., the lattice generated by the embeddings into $\Z^n$ of all shortest non-zero codewords in $\C$.

Notice that since $\bar{S}_{\C} \subset \lat$, $\lat$ certainly has at least $\kappa_0(\C)$ vectors with length $\sqrt{d(\C)}$. So, if we could somehow prove that $\lambda_1(\lat)^2 = d(\C)$, then we would have $\kappa_2(\lat) \geq \kappa_0(\C)$ and could thus convert a code with large kissing number into a lattice with large kissing number. However, it is unfortunately not the case that $\lambda_1(\lat)^2 = d(\C)$ in general, as we show in \cref{thm:all-shortest-counterexample} below.
Note that this is essentially the opposite issue as the one with Construction-D lattices $\lat_D$, which are constructed so that $\lambda_1(\lat_D)^2 = d(\C_a)$, but, as we showed in the previous section, need not contain the representatives $\bar{\vec{c}}$ of minimum-weight non-zero codewords $\vec{c} \in \C$.

We now give sufficient conditions for $\lat = \lat(\bar{S}_{\C})$ \emph{not} to have $\lambda_1(\lat)^2 = d(\C)$.

\begin{theorem} \label{thm:all-shortest-counterexample}
Let $A = (\vec{a}_1,\ldots,\vec{a}_\ell) \in \F_2^{n_1 \times \ell}$, $B = (\vec{b}_1,\ldots, \vec{b}_\ell) \in \F_2^{n_2 \times \ell}$, and $\vec{z} \in \Z^{\ell}$ satisfy the following properties. (Notice that these properties imply that $A$ and $B$ are \emph{not} full-rank matrices.) 
\begin{enumerate}
    \item \label{item:short_generators} 
    $d(\C(A)) = \|\vec{a}_i\|_0$ and $d(\C(B)) = \|\vec{b}_i\|_0$ for all $i \in \set{1, \ldots, \ell}$.
    \item \label{item:kernel_subset} $\ker(B) \subseteq \ker(A)$.
    \item \label{item:long_outside_kernel} For all $\vec{x} \in \ker(A) \setminus \ker(B)$, $\|B\vec{x}\|_0 > d(\C(B))$.
    \item \label{item:funny_kernel} $\bar{B} \vec{z} = \vec0$ but $\bar{A} \vec{z} \neq \vec0$.\footnote{Of course, \cref{item:kernel_subset} implies that $\bar{A} \vec{z}$ is zero \emph{modulo $2$}.}
\end{enumerate}
Then, for any sufficiently large positive integer $m$, the code
\[
    \C_m := \C\begin{pmatrix}
        A\\
        B \otimes \vec{1}_m
    \end{pmatrix}
\]
satisfies
\[
    \lambda_1(\lat(\bar{S}_{\C_m}))^2 < d(\C_m)
    \; .
\]
\end{theorem}
\begin{proof}
    We take $m$ to be large enough that $m \geq d(\C(A))$ \emph{and}
    \[
        d(\C(A)) + m \cdot d(\C(B)) > \|\bar{A}\vec{z}\|_2^2
        \; .
    \]
    Let 
    \[
        G_m := \begin{pmatrix}
        A\\
        B \otimes \vec{1}_m
    \end{pmatrix} \in \F_2^{(n_1 + m n_2) \times \ell}
    \; .
    \]
    
    We first argue that $d(\C_m) = d(\C(A)) + m \cdot d(\C(B))$. It follows from \cref{item:short_generators} that $d(\C_m) \leq d(\C(A)) + m \cdot d(\C(B))$. So, suppose that $\vec{x} \in \F_2^\ell$ is such that $G_m\vec{x} \neq \vec0$. If $A \vec{x} \neq \vec0$, then by \cref{item:kernel_subset}, it follows that $B\vec{x} \neq \vec0$ and therefore that $\|G_m \vec{x}\|_0 \geq d(\C(A)) + m \cdot d(\C(B))$. If $A\vec{x} =\vec0$, then \cref{item:long_outside_kernel} implies that 
    \[\|G_m\vec{x}\|_0 = m \|B\vec{x}\|_0 \geq m \cdot (d(\C(B)) + 1) \geq  d(\C(A)) + m \cdot d(\C(B))
    \; ,
    \]
    where in the second inequality we have used the fact that $m \geq d(\C(A))$. It follows that $d(\C_m) = d(\C(A)) + m \cdot d(\C(B))$ as claimed.

    Together with \cref{item:short_generators}, this implies that the columns of $G_m$ are shortest non-zero codewords in $\C_m$, i.e., 
    \[
        \begin{pmatrix}
            \vec{a}_i\\
            \vec{b}_i \otimes \vec{1}_m
        \end{pmatrix} \in S_{\C_m}
    \]
    for all $i$. Therefore, 
    \[
         \lat(\bar{G}_m) \subseteq \lat(\bar{S}_{\C_m})
        \; .
    \]

    So, it suffices to find a non-zero vector in $\lat(\bar{G}_m)$ with squared norm strictly less then $d(\C(A)) + m \cdot d(\C(B))$. Let $\vec{z} \in \Z^\ell$ be as in \cref{item:funny_kernel}. Since $\bar{A}\vec{z} \neq \vec0$ by assumption, we have $\bar{G}_m \vec{z} \neq \vec0$. On the other hand, since $\bar{B} \vec{z} = \vec0$, we have that
    \[
        \|\bar{G}_m \vec{z}\|_2 = \|\bar{A}\vec{z}\|_2
        \; .
    \]
    Since we chose $m$ such that 
    $
        d(\C(A)) + m \cdot d(\C(B)) > \|\bar{A}\vec{z}\|_2^2
    $,
    the result follows.
\end{proof}

By instantiating $A$, $B$, and $\vec{z}$ appropriately, we immediately obtain the following corollary. (More concretely, by choosing $m = 4$, we obtain a counterexample code with block length $18$, dimension $3$, and minimum distance $9$.)

\begin{corollary} \label{cor:all-shortest-counterexample}
There exists a (binary, linear) code $\C$ such that $\lambda_1(\lat(\bar{S}_{\C}))^2 < d(\C)$.
\end{corollary}
\begin{proof}
    Let
\[
 A := \begin{pmatrix}
            1 &1 &0 &0\\
            0&0&1&1
 \end{pmatrix} \text{ ,} \qquad
 B := \begin{pmatrix}
     1 &0 &1 &0\\
     0 &1 &0 &1\\
     0& 1 &1 &0\\
     1 &0 &0 &1
 \end{pmatrix} \text{ ,} \qquad
 \vec{z} := \begin{pmatrix}
 1\\1\\-1\\-1 \end{pmatrix} \ \text{.}
\]
One can then check that $A$, $B$, and $\vec{z}$ satisfy the conditions for \cref{thm:all-shortest-counterexample}, and the result follows.
\end{proof}

\begin{remark}
    We note that~\cite{Vladut-lpkissnum-2021} uses the construction described above to attempt to construct lattices with exponential kissing number in many different norms, and not just the $\ell_2$ norm. \cref{thm:all-shortest-counterexample} extends more-or-less immediately to $\ell_p$ norms for finite $p$. But, we also note that the binary Golay code yields a counterexample for all $p < 2$. The binary Golay code is perhaps a more natural counterexample because it does not have repeated coordinates.
\end{remark}

\subsection{\texorpdfstring{Constructions $C$, $\bar{D}$, and More}{C, the Code Formula, and More}}

Finally, we consider constructions of lattices based on towers of codes that are closed under Schur product.
The \emph{Schur product} or coordinate-wise product of two vectors $\vec{x}, \vec{y} \in \F_2^n$ is defined as $\vec{x} \odot \vec{y} = (x_1 y_1, \ldots, x_n y_n)^T \in \F_2^n$.
A tower of binary codes $\F_2^n = \C_0 \supseteq \C_1 \supseteq \cdots \supseteq \C_a$ is \emph{closed under the Schur product} if for every $i = 1, \ldots, a$ and pair of codewords $\vec{c}, \vec{c}' \in \C_i$, $\vec{c} \odot \vec{c}' \in \C_{i - 1}$. 

We note that this property is quite natural in the context of Construction-D-like constructions of lattices from a tower of codes. 
Indeed, for any $\vec{c}, \vec{c}' \in \F_2^n$, $\bar{\vec{c} + \vec{c}'} = \bar{\vec{c}} + \bar{\vec{c}'} - 2 \cdot (\bar{\vec{c} \odot \vec{c}'})$.
So, a tower of binary codes being closed under the Schur product implies that if $\vec{c}, \vec{c}' \in \C_i$ then $\bar{\vec{c} + \vec{c}'} = \bar{\vec{c}} + \bar{\vec{c}'} - 2 \cdot (\bar{\vec{c} \odot \vec{c}'}) \in \bar{C}_i - 2 \bar{C}_{i - 1} \subseteq \Z^n$.

\paragraph{Constructions in~\cite{vladut24}.}
Let $\C \subseteq \F_2^n$ be a code.~\cite{vladut24} defines the Construction $C^*$ lattice obtained from $\C$ as
\[
\lat_{C^*}(\C) := 2^n \cdot \Z^n + \bigcap_{i = 1}^n (2^{n - i} \cdot \bar{\C} + 2^{n - i + 1} \cdot \Z^n) \ \text{.}
\]
The definition of Construction $C^*$ is apparently new to~\cite{vladut24}.
We note that $\lat_{C^*}(\C)$ is in fact a lattice, but that it is simply a scaling of the Construction A lattice $\lat_A(\C)$. 
Indeed, notice that
\[
    \lat_{C^*}(\C) = 2^n \cdot \Z^n + \bigcap_{i=1}^n 2^{n-i} (\bar{\C} + 2\Z^n) = 2^n \cdot \Z^n + 2^{n-1} (\bar{\C} + 2\Z^n) = 2^{n-1} \cdot \lat_{A}(\C)
    \; ,
\]
since $2^{n-i} (\bar{\C} + 2\Z^n) \subseteq 2^{n-(i+1)} (\bar{\C} + 2\Z^n)$ for all $i$.

As noted in \cref{sec:cons-A}, Construction A lattices $\lat := \lat_A(\C) \subseteq \Z^n$ are such that $\lambda_1(\lat) \leq 2$, and so $\lat_A(\C)$ cannot have exponential kissing number (regardless of $\C)$. This property is of course preserved under scaling.
Because of this issue~\cite{vladut24} is invalid.

However,~\cite{vladut24} cites the sphere packing construction of Bos~\cite{bos80} (a variant of Construction C) and in personal communication~\cite{vladut-com-24} said that using the construction in~\cite{bos80} was his intention.
In particular,~\cite{vladut24} references~\cite[Theorem 1]{bos80}. However,~\cite[Theorem 1]{bos80} notes that for its sphere packing construction to yield a lattice, the tower of codes to which it is applied must be closed under the Schur product. 
It is not known whether suitable codes exist to produce a lattice with exponentially large kissing number from this construction.

Moreover,~\cite{vladut24} apparently uses a tower of codes $\C_1 = \cdots = \C_n = \C$ all of which are equal to a fixed code $\C \subseteq \F_2^n$. For such a tower of codes to be closed under the Schur product, it must therefore be the case that $\C$ itself is closed under the Schur product, i.e., that for every pair of codewords $\vec{c}, \vec{c}' \in \C$, $\vec{c} \odot \vec{c}' \in \C$. This in turn implies that any pair of distinct codewords $\vec{c}, \vec{c}' \in \C$ with $\norm{\vec{c}}_0 = \norm{\vec{c}'}_0 = d(\C)$ must have disjoint supports (since otherwise $\vec{c} \odot \vec{c}'$ would be a non-zero codeword in $\C$ with $\norm{\vec{c} \odot \vec{c}'} < d(\C)$). It is impossible to have more than $n$ non-zero codewords with pairwise disjoint supports, and so $\C$ cannot have exponential kissing number.

\paragraph{Construction $\bar{D}$.}
We conclude by discussing Construction $\bar{D}$, which is also known as the Code Formula or Forney's Code Formula~\cite{journals/tit/Forney88a,journals/tit/Forney88b,forney00,KositwattanarerkOggier14}. Let $\F_2^n = \C_0 \supseteq \C_1 \supseteq \cdots \supseteq \C_a$ be a tower of binary codes, and define
\begin{equation} \label{eq:L-d-bar}
\lat_{\bar{D}}(\C_1, \ldots, \C_a) := 2^a \cdot \Z^n + 2^{a - 1} \cdot \bar{\C}_1 + \cdots + 2 \cdot \bar{\C}_{a - 1} + \bar{\C}_a \ \text{.}
\end{equation}
It is known that $\lat_{\bar{D}}$ is a lattice if and only if the tower of codes $\F_2^n = \C_0 \supseteq \C_1 \supseteq \cdots \supseteq \C_a$ is closed under the Schur product, and that in this case it also coincides with Construction D; see~\cite[Theorem 1]{KositwattanarerkOggier14}.
Construction $\bar{D}$ is perhaps the cleanest of several related constructions of lattices from towers of codes that are closed under the Schur product.
In particular, although we have noted that the proof in \cite{vladut24} does not succeed, we do not rule out a similar approach that applies Construction $\bar{D}$ (or a related construction) to a tower of codes that is closed under the Schur product.

\end{document}